\documentclass[11pt, a4paper]{article}

\usepackage[margin=1in]{geometry} 
\usepackage{amsfonts, amscd, amssymb, mathtools, mathrsfs, dsfont, bbm, bbding} 
\usepackage[amsmath, amsthm, thmmarks]{ntheorem} 
\usepackage{graphicx, xypic, color, float} 
\usepackage{indentfirst}
\usepackage{lmodern}
\usepackage[T1]{fontenc} 
\usepackage{enumerate, listings, verbatim, paralist}
\usepackage{setspace, xspace}
\usepackage[colorlinks=true, citecolor=blue]{hyperref}
\usepackage{extarrows}
\usepackage{pdfpages}
\usepackage{ulem}

\usepackage{setspace}
\AtBeginDocument{%
\addtolength\abovedisplayskip{-0.15\baselineskip}%
\addtolength\belowdisplayskip{-0.15\baselineskip}%
\addtolength\abovedisplayshortskip{-0.15\baselineskip}%
\addtolength\belowdisplayshortskip{-0.15\baselineskip}%
}

\newtheorem{thm}{Theorem} [section]
\newtheorem{prop}[thm]{Proposition}

\newtheorem{lemma}[thm]{Lemma}

\newtheorem{assmp}{Assumption}[section]
\newtheorem{remark}{Remark}[section]

\numberwithin{equation}{section} 

\renewcommand{\geq}{\geqslant}
\renewcommand{\leq}{\leqslant}

\newcommand{\opfont}{\mathbb}

\newcommand{\E}{{\mathbb E}}
\newcommand{\BE}[2][]{\ensuremath{\operatorname{\opfont{E}}^{#1}\!\left[#2\right]}}

\newcommand{\R}{\ensuremath{\operatorname{\mathbb{R}}}}

\newcommand{\argmin}{\ensuremath{\operatorname{argmin}}}

\newcommand{\dd}{\ensuremath{\operatorname{d}\! }}
\newcommand{\dt}{\ensuremath{\operatorname{d}\! t}}
\newcommand{\de}{\ensuremath{\operatorname{d}\! e}}
\newcommand{\ds}{\ensuremath{\operatorname{d}\! s}}

\newcommand{\dw}{\ensuremath{\operatorname{d}\! W}}

\newcommand{\nn}{\nonumber}

\newcommand{\ol}{\mathcal{L}}

\newcommand{\ou}{\mathcal{U}}

\newcommand{\ito}{It\^{o}'s lemma\xspace} 

\newcommand{\cE}{\ensuremath{\mathcal{E}}}

\usepackage{comment}
\excludecomment{extra}

\begin{document}

\title{Mean-variance portfolio selection in jump-diffusion model under no-shorting constraint: A viscosity solution approach}
\author{Xiaomin Shi\thanks{School of Statistics and Mathematics, Shandong University of Finance and Economics, Jinan 250100, China. 
Email: \url{shixm@mail.sdu.edu.cn}}
\and Zuo Quan Xu\thanks{Department of Applied Mathematics, The Hong Kong Polytechnic University, Kowloon, Hong Kong, China. Email: \url{maxu@polyu.edu.hk}}}
\maketitle
\begin{abstract}
This paper concerns a continuous time mean-variance (MV) portfolio selection problem in a jump-diffusion financial model with no-shorting trading constraint. The problem is reduced to two subproblems: solving a stochastic linear-quadratic (LQ) control problem under control constraint, and finding a maximal point of a real function. Based on a two-dimensional fully coupled ordinary differential equation (ODE), we construct an explicit viscosity solution to the Hamilton-Jacobi-Bellman equation of the constrained LQ problem. Together with the Meyer-It\^o formula and a verification procedure, we obtain the optimal feedback controls of the constrained LQ problem and the original MV problem, which corrects the flawed results in some existing literatures. In addition, closed-form efficient portfolio and efficient frontier are derived. In the end, we present several examples where the two-dimensional ODE is decoupled.
\end{abstract} 

\addcontentsline{toc}{section}{\hspace*{1.8em}Abstract}

\section{Introduction.}
Research on mean-variance (MV) portfolio selection dates back to Markowitz \cite{Ma} in the single-period setting. By embedding the problem into a tractable auxiliary stochastic linear-quadratic optimal control problem, Li and Ng \cite{LN} and Zhou and Li \cite{ZL} made a breakthrough on MV problems, respectively, in multi-period and continuous-time settings. Since then, a large amount of research along this line has been conducted in more complicated and incomplete financial markets. For instance, Bielecki, Jin, Pliska and Zhou \cite{BJPZ} solved the MV problem with bankruptcy prohibition, Lim \cite{Lim} allowed for jumps in the underlying assets, Zhou and Yin \cite{ZY} featured assets in a regime switching market, Hou and Xu \cite{HX16} incorporated intractable claims in the MV problem.

In particular, Li, Zhou and Lim \cite{LZL} solved the MV problem under no-shorting constraint when all the prices of the assets are continuous. Through a conjecture-verification procedure, they constructed an explicit viscosity solution to the corresponding Hamilton-Jacobi-Bellman (HJB) equation via two uncorrelated ordinary differential equations (ODEs), called Riccati equations. Technically, they divided the whole space into three disjoint regions, then constructed smooth solutions in the first and third regions, and finally obtained a piecewise smooth function by sticking the three pieces together. They found that the optimal wealth process will stay in the initial region all the time. The same silent feature was also implicitly captured in Hu and Zhou \cite{HZ} when they studied MV problem with no-shoring constraints and random drift and diffusion coefficients. Why does the above phenomenon happen? To our own understanding, it is because that the underlying assets' prices and thus the wealth process are continuous processes.

In a jump-diffusion financial market, there are several papers studied MV problems via viscosity solution approach, even incorporating no-shorting constraint and reinsurance strategies; see, e.g., Bi and Guo \cite{BG}, Bi, Liang and Xu \cite{BLX}, Liang et al. \cite{LBYZ}. But the results in these papers are possibly not correct, which motives us to perform this study. Different from diffusion models such as \cite{LZL}, the HJB equation in a jump model is no longer a partial differential equation (PDE), but a partial integro-differential equation (PIDE). Note that the value function $V(t,x)$ maybe a piecewise function\footnote{Namely, it may take different forms in different regions.}, so $V(t,x+\pi\beta)$ and $V(t,x)$, both of which appeared in the PIDE, will possibly take different forms when the state jumps from its current value $x$ to a new value $x+\pi\beta$. Unfortunately, none of \cite{BG}, \cite{BLX} and \cite{LBYZ} took account of this possibility. Their method was first to divide the $(t,x)$-space into two regions separated by a known boundary, and then try to construct a viscosity solution to the related PIDE for each region separately, via two uncorrelated ODEs. Their arguments were indeed based on the unconvincing ansatz that the state values $x+\pi\beta$ (after jump) and $x$ (before jump) will always stay in the same region all the time so that the viscosity solution remains the same form in that region. Unfortunately, this is not the case and their constructed solutions were not able to be verified to be the original value functions. To the best of our knowledge, the viscosity solution of the PIDE of the MV problem in a jump-diffusion market under no-shorting constraint has never been addressed correctly in the existing literatures. This paper aims to fill this gap by deducing the right answer with solid proofs.

We first reduce the original MV problem to a homogenous stochastic LQ control problem under control constraint. By exploiting the \textit{ad hoc} structure of the HJB equation (which is actually a PIDE) of the stochastic LQ control problem, we derive, via heuristic argument, a candidate viscosity solution in terms of a two-dimensional fully coupled ODE. We succeed in showing that the coupled ODE admits a classical solution, which ensures that the constructed solution is indeed a viscosity solution to the HJB equation of the stochastic LQ control problem. Thanks to the existing result on the uniqueness of the viscosity solution, the constructed solution is nothing but the value function of the LQ problem. Further, using Meyer-It\^o formula and a rigorous verification procedure, we obtain the optimal feedback control of the MV problem as well as the efficient frontier (which is a half-line).

Compared with existing works such as \cite{LZL}, the salient distinctive feature of this paper is that the corresponding wealth $X$ will probably up/down cross the vertex point $d^{*}$ at the jump time of the underlying Poisson random measure. This essentially renders the associated two-dimensional ODE \textit{fully coupled}, and thus cannot be solved one by one, which repudiates the solutions in \cite{BG}, \cite{BLX} and \cite{LBYZ}.

The remainder of this paper is organized as follows. In Section \ref{sec:pf}, we formulate an MV
portfolio problem in a jump-diffusion financial market under short-selling prohibition. Section \ref{sec:viscosity} constructs a viscosity solution in terms of a fully coupled two-dimensional ODE, to the related HJB equation of a stochastic LQ control problem. In Section \ref{sec:solution}, we obtain the efficient investment strategies and efficient frontier of the original constrained MV problem.
Section \ref{sec:Solvability} is devoted to the solvability of the corresponding coupled two-dimensional ODE.
In Section \ref{sec:examples}, we discuss some special cases where the two-dimensional ODE can be decoupled, hence the efficient portfolio possess more explicit formula.

\section{Problem formulation.}\label{sec:pf}
Let $(\Omega, \mathcal F, \mathbb{F}, \mathbb{P})$ be a fixed complete filtered probability space, on which is defined two independent random processes:
a standard $n$-dimensional Brownian motion $W_t=(W_{1,t}, \ldots, W_{n,t})^{\top}$, and
an $\ell$-dimensional Poisson random measure $N(\dt,\de)$ defined on $\R_+\times\:\cE$ with a stationary compensator (intensity measure) given by $\nu(\de)\dt=(\nu_1(\de),\ldots,\nu_{\ell}(\de))^{\top}\dt$ satisfying $\sum_{j=1}^{\ell}\nu_j(\cE)<\infty$, where $\mathcal{\cE}\subseteq \R^{\ell}\setminus\{0\}$ is a nonempty Borel subset of $\R^{\ell}$ and $\mathcal{B}(\cE)$ denotes the Borel $\sigma$-algebra generated by $\cE$.
The compensated Poisson random measure is denoted by $ \widetilde N(\dt,\de)$.
Throughout the paper, we let $T$ be a fixed positive constant to stand for the investment horizon.

We denote by $\R^\ell$ the set of all $\ell$-dimensional column vectors, by $\R^\ell_+$ the set of vectors in $\R^\ell$ whose components are nonnegative, by $\R^{\ell\times n}$ the set of $\ell\times n$ real matrices, and by $\mathbb{S}^n$ the set of symmetric $n\times n$ real matrices. Therefore, $\R^\ell\equiv\R^{\ell\times 1}$. For any vector $Y$, we denote $Y_i$ as its $i$-th component. For any vector or matrix $M=(M_{ij})$, we denote its transpose by $M^{\top}$, and its norm by $|M|=\sqrt{\sum_{ij}M_{ij}^2}$. If $M\in\mathbb{S}^n$ is positive definite (resp. positive semidefinite), we write $M>$ (resp. $\geq$) $0.$ We write $A>$ (resp. $\geq$) $B$ if $A, B\in\mathbb{S}^n$ and $A-B>$ (resp. $\geq$) $0.$ We use the standard notations $x^+=\max\{x, 0\}$ and $x^-=\max\{-x, 0\}$ for $x\in\R$. All the equations and inequalities in subsequent analysis shall be understood in the sense that $\dd\mathbb{P}$-a.s. or $\dd\nu$-a.e. or $\dt$-a.e. or $\dt\otimes \dd\nu$-a.e. etc, which shall be seen from their expressions.

Suppose that there are one money account with zero interest rate\footnote{As is well-known, there is no essential difference if the interest is not zero but a deterministic function of time, since one can discount everything without change the nature of the problem. Hence we assume the interest is zero in this paper. } and $m$ tradable risky assets (e.g., stocks or insurance claims). For $i=1,\ldots,m$, we assume that $S_{i,t}$, the price of the $i$-th risky asset, satisfies the following stochastic differential equations (SDEs) with jumps:
\begin{align*}
\dd S_{i,t}=S_{i,t-}\Big(\mu_{i,t}\dt+\sum_{j=1}^{n}\sigma_{ij,t}\dw_{j,t}+\sum_{j=1}^{\ell}\int_{\cE}\beta_{ij,t}(e)\tilde N_j(\dt,\de)\Big), \ S_{i,0}=s_{i}>0.
\end{align*}
Denote $\mu_t:=(\mu_{1,t},\ldots,\mu_{m,t})^{\top}$, $\sigma_{t}:=(\sigma_{ij,t})_{m\times n}$, $\beta_{t}(e):=(\beta_{ij,t}(e))_{m\times \ell}$, and $\beta_{j,t}$ the $j$-th column of $\beta_{t}$, $j=1,\ldots,\ell$.
Throughout the paper, we put the following assumption on the market parameters without claim.
\begin{assmp}\label{assu1}
For all $i=1,\ldots,m$, $k=1,\ldots,n$, $j=1,\ldots,\ell$ the coefficients
$\mu_{i}$, $\sigma_{ik}$ (resp. $\beta_{ij}$) are deterministic, Borel-measurable and bounded functions on $[0,T]$ (resp. $[0,T]\times\cE$) and $\beta_{ij}>-1$\footnote{The assumption $\beta_{ij}>-1$ is used to ensure that the stock price is always positive. If one relaxes it to $\beta_{ij}\geq -1$, then the stock price can be zero (and will stay there forever once it hits zero), meaning that bankruptcy happens. Our argument still works without any changes. }. And there exists a constant $\delta>0$ such that
\[
\Sigma_t:=\sigma_t\sigma_t^{\top}+\sum_{j=1}^{\ell}\int_{\cE}\beta_{j,t}\beta_{j,t}^{\top}\nu_j(\de)
\geq\delta\mathbf{1}_m, \ \forall t\in[0,T],
\]
where $\mathbf{1}_{m}$ denotes the $m$-dimensional identity matrix.
\end{assmp}

We consider
a small investor, whose actions cannot affect the asset prices. He will decide at every time
$t\in[0, T]$ the amount $\pi_{i,t}$ of his wealth to invest in the $i$-th risky asset, $i=1, \ldots, m$. The vector process $\pi:=(\pi_1, \ldots, \pi_m)^{\top}$ is called a portfolio of the investor. Then the investor's self-financing wealth process $X$ corresponding to a portfolio $\pi$ is the strong solution to the SDE:
\begin{align}
\label{wealth}
\begin{cases}
\dd X_t=\pi_t^{\top}\mu_t \dt+\pi_t^{\top}\sigma_t\dw_t+\int_{\cE}\pi_t^{\top}\beta_t(e) \widetilde N(\dt,\de), \\
X_0=x.
\end{cases}
\end{align}
Denote by $X^{\pi}$ the wealth process \eqref{wealth} whenever it is necessary to indicate its dependence on the portfolio $\pi$.

The admissible portfolio set is defined as\footnote{In our model all the stocks are prohibited to be short sale. There is no essential difficulty to extend our model to the case that some of the stocks are allowed to be short sale.}
\begin{align*}
\mathcal U=\Big\{\pi:[0,T]\times\Omega\to\R_{+}^{m}\;\Big|\; \pi \ \mbox{is predictable}, \ \E\int_0^T\pi_t^2\dt<\infty \Big\}.
\end{align*}
For any $\pi\in \mathcal{U}$, the SDE \eqref{wealth} has a unique strong solution $X$.

\par
For a given expectation level $z\geq x$, the investor's MV problem is to
\begin{align}
\mathrm{Minimize}&\quad \mathrm{Var}(X_T^{\pi})\equiv\E[(X_T^{\pi})^{2}-z^2]%
, \nn\\
\mathrm{ s.t.} &\quad
\begin{cases}
\E[X_T^{\pi}]=z, \\
\pi\in \mathcal{U}.
\end{cases}
\label{optm}%
\end{align}
Denote $\Pi_{z}=\{\pi\;|\;\pi\in\ou~ \mbox{and} ~\E [X_T^{\pi}]=z\}$.
The MV problem \eqref{optm} is called feasible if
$\Pi_{z}$ is not empty. Correspondingly, any $\pi\in\Pi_{z}$ is called a feasible/admissible portfolio to the MV problem \eqref{optm}. A portfolio $\pi\in\Pi_{z}$ is called optimal to the MV problem \eqref{optm} if the minimum is achieved at $\pi$.

Trivially one can verify that $\pi=0$ is an optimal portfolio to the MV problem \eqref{optm} when $z=x$. The case $z<x$ is financially meaningless as any rational investor would expect a better return than doing nothing,
hence, we will mainly focus on the nontrivial case $z>x$ from now on. Our first result resolves the feasibility issue of the MV problem \eqref{optm} in the case $z>x$.

\begin{lemma}
The following three statements are equivalent:
\begin{itemize}
\item The problem \eqref{optm} is feasible for some $z> x$;
\item The problem \eqref{optm} is feasible for all $z> x$;
\item It holds that \begin{align}\label{feasible}
\sum_{i=1}^m\int_0^T\mu_{i,t}\dt>0.
\end{align}
\end{itemize}
\end{lemma}
The proof is exactly the same as \cite[Theorem 5.3]{HSX}, so we omit it.
In the rest of this paper, we will always assume Assumption \ref{assu1} and \eqref{feasible} hold.

An optimal strategy $\pi^{\ast}$ to \eqref{optm}
is called an optimal strategy corresponding to the level $z$. And $\Big(\sqrt{\mathrm{Var}(X_{T}^{\pi^\ast})}, z\Big)$ is
called an efficient point of the MV problem \eqref{optm}. The set of efficient points $$\Big\{\Big(\sqrt{\mathrm{Var}(X_{T}^{\pi^\ast})}, z\Big)\;\Big|\; z\geq x\Big\}$$ is
called the efficient frontier of the MV problem \eqref{optm}. This paper will derive the explicit optimal portfolio for each $z\geq x$ and explicit efficient frontier of the MV problem \eqref{optm}.

\section{Viscosity solution to HJB equation.}\label{sec:viscosity}
The way to solve the MV problem \eqref{optm} is rather clear nowadays. To deal with the budget constraint $\E[X_T]=z$, we introduce a Lagrange
multiplier $-2d\in\R$ and introducing the following standard stochastic LQ control
problem:
\begin{align}\label{optmun}
V(0,x;d):=\inf_{\pi\in\mathcal{U}}{\mathbb{E}}[(X_T-d)^{2}],~~(x,d)\in \R^{2}.
\end{align}

Now define a homogenous stochastic LQ control
problem:
\begin{align}\label{prob1}
\varphi(t,x):=\inf_{\pi\in\ou} \BE{X_T^2\;\Big|\;X_t=x},~~(t,x)\in[0,T]\times\R.
\end{align}
By taking $\pi \equiv 0$, we obtain $\varphi(t,0)\equiv 0$.
One can easily to verify that
\begin{align}\label{valuechange}
V(0,x;d)=\varphi(0,x-d),~~(x,d)\in\R^{2}.
\end{align}
Also, any optimal control to $\varphi(0,x-d)$ is optimal to $V(0,x;d)$, and vise versa.

According to the Lagrange duality theorem (see Luenberger \cite{Lu}), the problem \eqref{optm} is linked to the problems \eqref{optmun} and \eqref{prob1} by
\begin{align}\label{duality}
\inf_{\pi\in\Pi_{z}}\mathrm{Var}(X_T)&=\sup_{d\in\R}[V(0,x;d)-(d-z)^{2}].
\end{align}
So we can solve the MV problem \eqref{optm} by a two-step procedure: Firstly determine the function $V$, and then try to find a $d^{*}$ to maximize $d\mapsto V(0,x;d)-(d-z)^{2}$.

In this paper, we adopt the viscosity approach to study the problem \eqref{prob1}.
\begin{lemma}\label{viscosity}
The value function $\varphi$ defined in \eqref{prob1} is the unique quadratic growth (w.r.t. the spacial argument) viscosity solution to the following HJB equation:
\begin{align}\label{HJB2}
\begin{cases}
&\varphi_t(t,x)+\inf\limits_{\pi\in\R^m_+}\ol_{\pi}\varphi(t,x)=0,~~(t,x)\in[0,T)\times\R,\\
&\varphi(T,x)=x^2,
\end{cases}
\end{align}
where the integro-differential operator $\mathcal{L}_{\pi}$ is defined by
\begin{align*}
\ol_{\pi}\phi(t,x)&:=\frac{1}{2}\phi_{xx}(t,x)|\pi^{\top}\sigma_t|^2+\phi_x(t,x)\pi^{\top}\mu_t\\
&\qquad+\sum_{j=1}^{\ell}\int_{\cE}[\phi(t,x+\pi^{\top}\beta_{j,t}(e))
-\phi(t,x)-\phi_x(t,x)\pi^{\top}\beta_{j,t}(e)]\nu_j(\de),
\end{align*}
for $(t,x,\pi)\in[0,T]\times\R\times\R^m_+$ and $\phi:[0,T]\times\R\to\R$.
\end{lemma}
\begin{proof}
This follows from \cite[Theorem 4.1, Theorem 5.1]{BHL}.
\end{proof}
By this result, solving the problem \eqref{prob1} is reduced to find a viscosity solution to \eqref{HJB2}.

\subsection{Viscosity solution to \eqref{HJB2}: A heuristic derivation.}\label{heuristic}
To find a viscosity solution to \eqref{HJB2},
a key observation is that the value function to \eqref{prob1} enjoys quadratic positive homogeneity, which implies, by Lemma \ref{viscosity}, that the viscosity solution to \eqref{HJB2} enjoys the same property.

\begin{lemma}\label{lemmahomo}
The function $\varphi$ defined in \eqref{prob1} satisfies
\begin{align}\label{homo}
\varphi(t,\lambda x)=\lambda^2\varphi(t,x)\ \mbox{for any $(t,x,\lambda)\in[0,T]\times\R\times\R_{+}$}.
\end{align}
Consequently,
\begin{align}
\label{valuefun}
\varphi(t,x)=P_{+,t}(x^+)^2+P_{-,t}(x^-)^2,
\end{align}
where $P_{\pm,t}:=\varphi(t,\pm1)$ with $P_{\pm,T}=1$.
\end{lemma}
\begin{proof}
Clearly, the dynamics \eqref{wealth} is a linear homogenous system of $(X,\pi)$ and objective in \eqref{prob1} is a quadratic homogenous system of $(X,\pi)$, so the first conclusion follows. The second claim trivially follows by taking $\lambda=1/|x|$ for $x\neq 0$ in \eqref{homo}.
\end{proof}

According to Lemmas \ref{viscosity} and \ref{lemmahomo}, the unique viscosity solution to \eqref{HJB2} must take the form of \eqref{valuefun}.
Although we know $\varphi(t,0)\equiv 0$, but due to the jump terms $\varphi(t,x+\pi^{\top}\beta_j)-\varphi(t,x)$ appeared in the integro-differential operator $\ol_{\pi}\varphi$, we cannot solve the problem \eqref{HJB2} in the regions $x>0$ and $x<0$, separately. In other words, the solutions to \eqref{HJB2} in the regions $x>0$ and $x<0$ (that is, $P_{+,t}$ and $P_{-,t}$) are coupled together, we must solve them simultaneously. This is the essential difference between a continuous diffusion model and the diffusion-jump model. In the former case one can solve the problem in the two regions separately (see, e..g, \cite{HSX}, \cite{HZ}, \cite{LZL}).

Because the problem \eqref{prob1} is Markovian, both $P_{+,t}$ and $P_{-,t}$ must be deterministic function of $t$.
We now use the fact that $\varphi$ is the viscosity solution to \eqref{HJB2} to
derive the dynamics of $(P_{+,t}, P_{-,t})$ by intuitive argument.

We suppose that $(P_{+,t}, P_{-,t})$ is governed by a two-dimensional coupled ODE:
\begin{align}\label{Riccati}
\begin{cases}
\dot{P}_{+,t}=-H_{+}^*(t,P_{+,t},P_{-,t}),\\
\dot{P}_{-,t}=-H_{-}^*(t,P_{+,t},P_{-,t}),\\
P_{\pm,T}=1, \ P_{\pm,t}>0,
\end{cases}
\end{align}
for some functions $H_{+}^*$ and $H_{-}^*$ on $[0,T]\times\R_{+}\times\R_{+}$ to be determined. Note here we suspect $P_{\pm,t}=\varphi(t,\pm1)>0 $ because the market is free of arbitrage opportunity.

Under the above conjecture, we have $\varphi\in C^{1,2}([0,T)\times\R/\{0\})$, therefore, it satisfies \eqref{HJB2} in the classical sense in that region.
In particular, for $(t,x)\in [0,T)\times (0,\infty)$, we have $\varphi(t,x)=P_{+,t}x^{2}$ and
\[
\varphi(t,x+\pi^{\top}\beta_{j,t}(e))=P_{+,t}[(x+\pi^{\top}\beta_{j,t}(e))^+]^2
+P_{-,t}[(x+\pi^{\top}\beta_{j,t}(e))^-]^2.
\]
Remind that many existing works used a wrong form of the second equation\footnote{such as $\varphi(t,x+\pi^{\top}\beta_{j,t}(e))=P_{+,t}(x+\pi^{\top}\beta_{j,t}(e))^2$.} that eventually led to wrong solutions.
Taking the above expressions
into \eqref{HJB2} yields
\begin{align*}
&\quad\; \varphi_t(t,x)+\inf\limits_{\pi\in\R^m_+}\ol_{\pi}\varphi(t,x)\nn\\
&=\dot{P}_{+,t}x^2+\inf\limits_{\pi\in\R^m_+}\Big\{P_{+,t}|\pi^{\top}\sigma_t|^2+2P_{+,t} x\pi^{\top}\mu_t+\sum_{j=1}^{\ell}\int_{\cE}\Big[P_{+,t}[(x+\pi^{\top}\beta_{j,t}(e))^+]^2\nn\\
&\qquad\qquad\qquad\qquad~
+P_{-,t}[(x+\pi^{\top}\beta_{j,t}(e))^-]^2-P_{+,t}x^2-2P_{+,t}x\pi^{\top}\beta_{j,t}(e)\Big]\nu_j(\de)\Big\}\nn\\
&=x^2\Bigg\{\dot{P}_{+,t}+\inf\limits_{\pi\in\R^m_+}\Big\{P_{+,t}\big|(\pi/x)^{\top}\sigma_t\big|^2+2P_{+,t}(\pi/x)^{\top}\mu_t\nn\\
&\qquad\qquad+\sum_{j=1}^{\ell}\int_{\cE}\Big[P_{+,t}\Big[\big[\big(1+(\pi/x)^{\top}\beta_{j,t}(e)\big)^+\big]^2-2(\pi/x)^{\top}\beta_{j,t}(e)-1\Big]\nn\\
&\qquad\qquad\qquad\qquad~+P_{-,t}\big[\big(1+(\pi/x)^{\top}\beta_{j,t}(e)\big)^-\big]^2\Big]\nu_j(\de)\Big\}\Bigg\}\nn\\
&=x^2\Bigg\{-H_{+}^*(t,P_{+,t},P_{-,t})+\inf\limits_{v\in\R^m_+}\Big\{P_{+,t}|v^{\top}\sigma|^2+2P_{+,t} v^{\top}\mu_t\nn\\
&\qquad\qquad+\sum_{j=1}^{\ell}\int_{\cE}\Big[P_{+,t}\Big[\big[(1+v^{\top}\beta_{j,t}(e))^+\big]^2-1-2v^{\top}\beta_{j,t}(e)\Big]\nn\\
&\qquad\qquad\qquad\qquad
+P_{-,t}\big[(1+v^{\top}\beta_{j,t}(e))^-\big]^2\Big]\nu_j(\de)\Big\}\Bigg\}\nn\\
&=0,
\end{align*}
where we used the fact that $\pi\in\R^m_+$ if and only if $\pi/x\in\R^m_+$ since $x>0$.
Hence we shall have for $(t,P_{+},P_{-})\in [0,T]\times \R_{+}\times \R_{+}$,
\begin{align} \label{defH1}
H_{+}^*(t,P_{+},P_{-})&=
\inf\limits_{v\in\R^m_+}H_{+}(t,v,P_{+},P_{-}),
\end{align}
where
\begin{align}\label{def:Hv1}
H_{+}(t,v,P_{+},P_{-}) &:=P_{+}|v^{\top}\sigma|^2+2P_{+} v^{\top}\mu_t\nn\\
&\quad\;+\sum_{j=1}^{\ell}\int_{\cE}\Big[P_{+}[[(1+v^{\top}\beta_{j,t}(e))^+]^2-1-2v^{\top}\beta_{j,t}(e)]\nn\\
&\qquad\qquad\qquad
+P_{-}[(1+v^{\top}\beta_{j,t}(e))^-]^2\Big]\nu_j(\de).
\end{align}
Similarly, for $(t,x)\in [0,T)\times (-\infty,0)$, we can obtain
\begin{align} \label{defH2}
H_{-}^*(t, P_{+},P_{-})&=
\inf\limits_{v\in\R^m_+}H_{+}(t,v,P_{+},P_{-}),
\end{align}
where
\begin{align}\label{def:Hv2}
H_{-}(t,v,P_{+},P_{-}) &:=P_{-}|v^{\top}\sigma|^2-2P_{-} v^{\top}\mu_t\nn\\
&\quad\;+\sum_{j=1}^{\ell}\int_{\cE}\Big[P_{-}[[(1-v^{\top}\beta_{j,t}(e))^+]^2+2v^{\top}\beta_{j,t}(e)-1]^2\nn\\
&\qquad\qquad\quad+P_{+}[(1-v^{\top}\beta_{j,t}(e))^-]^2\Big]\nu_j(\de).
\end{align}
\begin{remark}
Trivially,
\begin{align*}
H^{*}_{\pm}(t,P_{+},P_{-}) \leq H_{\pm}(t, 0,P_{+},P_{-})=0.
\end{align*}
This will be used frequently below without claim.
\end{remark}

Now we have conjectured the expressions of $H_{+}^*$ and $H_{-}^*$. The following result resolves the solvability issue of the corresponding ODE \eqref{Riccati} for $(P_{+,t}, P_{-,t})$.
\begin{thm}\label{existence}
Let $H_{+}^*$ and $H_{-}^*$ be defined by \eqref{defH1} and \eqref{defH2}. Then the corresponding ODE \eqref{Riccati} admits a unique classical positive solution $(P_{+},P_{-})$. Furthermore,
\begin{align}\label{hatv}
\hat v_{\pm}(t,P_{+,t},P_{-,t})&:=\argmin_{v\in\R^m_+}H_{\pm}(t,v,P_{+,t},P_{-,t})
\end{align}
are bounded functions on $[0,T]$.
\end{thm}
The proof is slightly technical and we defer it to Section \ref{sec:Solvability}.

\begin{prop}\label{homeprop}
Let $(P_{+},P_{-})$ be given in Theorem \ref{existence}.
Then the value function of the problem \eqref{prob1} is given by
\begin{align}
\varphi(t,x)=P_{+,t}(x^+)^2+P_{-,t}(x^-)^2,~~(t,x)\in[0,T]\times\R.
\end{align}
\end{prop}
\begin{proof}
By Lemma \ref{viscosity}, it suffices to prove that
\begin{align}
\bar\varphi(t,x)\equiv P_{+,t}(x^+)^2+P_{-,t}(x^-)^2
\end{align}
is the unique quadratic growth (w.r.t. the spacial argument) viscosity solution to the HJB equation \eqref{HJB2}.

First, since $P_{\pm,T}=1$, the boundary condition $\bar\varphi(T,x)=x^{2}$ is satisfied.
Second, the aforementioned argument shows that, in the classical sense,
\begin{align}
\bar\varphi_t(t,x)+\inf\limits_{\pi\in\R^m_+}\ol_{\pi}\bar\varphi(t,x)=0, ~~(t,x)\in [0,T)\times\R/\{0\}.
\end{align}
Hence, it is only left to show $\bar\varphi$ is a viscosity solution to the HJB equation \eqref{HJB2} at any point $(t, 0)$ with $t\in[0,T)$.

Indeed, for any test function $\phi\in C^{1,2}([0,T]\times\R)$ such that the function $\phi-\bar\varphi$ arrives its minimum value 0 at the point $(t, 0)$, we have
\begin{align}
\phi(t,0)=\bar\varphi(t,0)=0,~~ \phi_{t}(t,0)\geq \bar\varphi_{t}(t,0)=0,~~
\phi_{x}(t,0)=\bar\varphi_{x}(t,0)=0.
\end{align}
If $\phi_{xx}(t,0)<2P_{+,t}$, then $\phi_{xx}(t,\theta)<2P_{+,t}$ whenever $\theta>0$ is sufficiently small. So, whenever $x>0$ is sufficiently small, by Talyor's expansion,
\begin{align*}
\phi(t,x)=\phi(t,0)+\phi_{x}(t,0)x+\frac{1}{2} \phi_{xx}(t,\theta)x^{2}=\frac{1}{2} \phi_{xx}(t,\theta)x^{2}
<P_{+,t} x^{2}=\bar\varphi(t,x),
\end{align*}
contradicting $\phi-\bar\varphi\geq 0$. Hence we proved $\phi_{xx}(t,0)\geq 2P_{+,t}$. Similarly, we can prove $\phi_{xx}(t,0)\geq 2P_{-,t}$, so $\phi_{xx}(t,0)\geq 2\max\{P_{+,t}, P_{-,t}\}$.
It then follows, for any $\pi\in\R^m_+$,
\begin{align*}
\phi_t(t,0)+\ol_{\pi}\phi(t,0)
&\geq\frac{1}{2}\phi_{xx}(t,0)|\pi^{\top}\sigma_t|^2+\phi_x(t,0)\pi^{\top}\mu_t\\
&\qquad+\sum_{j=1}^{\ell}\int_{\cE}[\phi(t, \pi^{\top}\beta_{j,t}(e))
-\phi(t,0)-\phi_x(t,0)\pi^{\top}\beta_{j,t}(e)]\nu_j(\de)\\
&\geq \max\{P_{+,t}, P_{-,t}\}|\pi^{\top}\sigma_t|^2+\sum_{j=1}^{\ell}\int_{\cE}\bar\varphi(t, \pi^{\top}\beta_{j,t}(e))\nu_j(\de)\\
&\geq 0,
\end{align*}
where we used fact that $\phi\geq \bar\varphi\geq 0$ to derive the last two inequalities.
This shows $\bar\varphi$ is a viscosity supsolution to \eqref{HJB2} at the point $(t, 0)$. Similarly one can show $\bar\varphi$ is also a viscosity subsolution to \eqref{HJB2} at the point $(t, 0)$, completing the proof.
\end{proof}

\section{Solutions to \eqref{optmun} and \eqref{optm}.}\label{sec:solution}
We now provide complete answers to the problems \eqref{optmun} and \eqref{optm} by verification arguments.

\subsection{Solution to \eqref{optmun}.}

\begin{prop}\label{verifi}
Let $(P_{+},P_{-})$, $\hat v_{\pm}(t,P_{+,t},P_{-,t})$ be given in Theorem \ref{existence}.
Then the optimal value of the LQ problem \eqref{optmun} is
\begin{align*}
V(0,x;d)=P_{+,0}\big[(x-d)^+\big]^2+P_{-,0}\big[(x-d)^-\big]^2.
\end{align*}
Moreover, the state feedback control defined by
\begin{align}\label{pistar}
\pi^*(t,X)=\hat v_{+}(t,P_{+,t},P_{-,t})(X_{t-}-d)^++\hat v_{-}(t,P_{+,t},P_{-,t})(X_{t-}-d)^-,
\end{align}
is optimal for the LQ problem \eqref{optmun}.
\end{prop}
\begin{proof}

The first assertion is a consequence of Proposition \ref{homeprop} and equation \eqref{valuechange}.
Substituting the feedback control \eqref{pistar} into the dynamics \eqref{wealth}, we get
\begin{align}
\begin{cases}
\dd Y_t=Y_{t-}^+\Big[\hat v_{+,t}^{\top}\mu_t\dt+\hat v_{+,t}^{\top}\sigma_t\dw_t+\int_{\cE} \hat v_{+,t}^{\top}\beta_t(e) \widetilde N(\dt,\de)\Big]\\
\qquad ~~+Y_{t-}^-\Big[\hat v_{-,t}^{\top}\mu_t\dt+\hat v_{-,t}^{\top}\sigma_t\dw_t+\int_{\cE} \hat v_{-,t}^{\top}\beta_t(e) \widetilde N(\dt,\de)\Big], \\
Y_0=y.
\end{cases}
\end{align}
where $Y_t:=X^{\pi^*}_t-d$, $y:=x-d$, and $\hat v_{\pm,t}:=\hat v_{\pm}(t,P_{+,t},P_{-,t})$.
Because the coefficients and $\hat v_{\pm,t}$ are bounded, the above SDE has a Lipschitz driver, hence admits a unique square integrable solution $Y$, which implies $\pi^*(t,X^{\pi^*}_{t})\in\ou$.

To verify the optimality of the feedback control \eqref{pistar},
it remains to prove
\begin{align*}\label{verify}
\E[(X^{\pi^*}_T-d)^2]&=P_{+,0}\big[(x-d)^+\big]^2+P_{-,0}\big[(x-d)^-\big]^2.
\end{align*}
By the Meyer-It\^o formula \cite[Theorem 70]{Protter}, we have
\begin{align*}
\dd Y_t^+&=Y_{t-}^+\Big[\big( \hat v_{+,t}^{\top}\mu_t-\sum_{j=1}^{\ell}\int_{\cE} \hat v_{+,t}^{\top}\beta_{j,t}\nu_j(\de)\big)\dt
+\hat v_{+,t}^{\top}\sigma_t\dw_t\Big]\\
&\qquad\qquad+\sum_{j=1}^{\ell}\int_{\cE}\Big[(Y_{t-}+Y_{t-}^+\hat v_{+,t}^{\top}\beta_{j,t}+Y_{t-}^- \hat v_{-,t}^{\top}\beta_{j,t})^+-Y_{t-}^+\Big]N_j(\dt,\de)+\frac{1}{2}\dd \mathbb{L}_t,
\end{align*}
and
\begin{align*}
\dd Y_t^-&=-Y_{t-}^-\Big[\big( \hat v_{-,t}^{\top}\mu_t-\sum_{j=1}^{\ell}\int_{\cE} \hat v_{-,t}^{\top}\beta_{j,t}\nu_j(\de)\big)\dt
+\hat v_{-,t}^{\top}\sigma_t\dw_t\Big]\\
&\qquad\qquad+\sum_{j=1}^{\ell}\int_{\cE}\Big[(Y_{t-}+Y_{t-}^+\hat v_{+,t}^{\top}\beta_{j,t}+Y_{t-}^- \hat v_{-,t}^{\top}\beta_{j,t})^--Y_{t-}^-\Big]N_j(\dt,\de)+\frac{1}{2}\dd \mathbb{L}_t,
\end{align*}
where $\mathbb{L}$ is the local time of $Y$ at $0$. Since $Y_t^{\pm}\dd \mathbb{L}_t=0$,
applying It\^{o} formula to $(Y_t^+)^2$ and $(Y_t^-)^2$ respectively yields
\begin{align*}
\dd\;(Y_t^+)^2&=(Y_{t-}^+)^2\Big[| \hat v_{+,t}^{\top}\sigma_t|^2+2r+2 \hat v_{+,t}^{\top}\mu_t-2\int_{\cE}v_{+,t}^{\top}\beta_t\nu(\de)\Big]\dt+2(Y_{t-}^+)^2 \hat v_{+,t}^{\top}\sigma_t\dw_t\\
&\qquad+\sum_{j=1}^{\ell}\int_{\cE}\Big[((Y_{t-}+Y_{t-}^+\hat v_{+,t}^{\top}\beta_{j,t}+Y_{t-}^- \hat v_{-,t}^{\top}\beta_{j,t})^+)^2-(Y_{t-}^+)^2
\Big]N(\dt,\de)\\
&=(Y_{t-}^+)^2\Big[| \hat v_{+,t}^{\top}\sigma_t|^2+2r+2 \hat v_{+,t}^{\top}\mu_t-2\int_{\cE}v_{+,t}^{\top}\beta_t\nu(\de)\Big]\dt+2(Y_{t-}^+)^2 \hat v_{+,t}^{\top}\sigma_t\dw_t\\
&\qquad+\sum_{j=1}^{\ell}\int_{\cE}\Big[(Y_{t-}^+)^2((1+\hat v_{+,t}^{\top}\beta_{j,e})^+)^2+(Y_{t-}^-)^2((-1+\hat v_{-,t}^{\top}\beta_{j,e})^+)^2-(Y_{t-}^+)^2
\Big]N(\dt,\de),
\end{align*}
and
\begin{align*}
\dd\;(Y_t^-)^2
&=(Y_{t-}^-)^2\Big[| \hat v_{-,t}^{\top}\sigma_t|^2+2r-2 \hat v_{-,t}^{\top}\mu_t+2\int_{\cE}v_{-,t}^{\top}\beta_t\nu(\de)\Big]\dt- 2(Y_{t-}^-)^2 \hat v_{-,t}^{\top}\sigma_t\dw_t\\
&\qquad+\sum_{j=1}^{\ell}\int_{\cE}\Big[(Y_{t-}^+)^2((1+\hat v_{+,t}^{\top}\beta_{j,e})^-)^2+(Y_{t-}^-)^2((-1+\hat v_{-,t}^{\top}\beta_{j,e})^-)^2-(Y_{t-}^-)^2
\Big]N(\dt,\de).
\end{align*}
Define a sequence of stopping times, for $n>0$,
\[
\tau_n=\inf\big\{t\geq 0\;\big|\;|Y_t|>n\big\}\wedge T.
\]
Recall the definitions of $H_{\pm}^*$, $H_{\pm}$ and $\hat v_{\pm}$,
and apply \ito to $P_{+,t}(Y_t^+)^2+P_{-,t}(Y_t^-)^2$ on $[0,\tau_n]$, we get
\begin{align*}
&\quad\;\E[P_{+,\tau_n}(Y_{\tau_n}^+)^2+P_{-,\tau_n}(Y_{\tau_n}^-)^2]\\
&=P_{+,0}(Y_{0}^+)^2+P_{-,0}(Y_{0}^-)^2\\
&\quad\;+\E\int_0^{\tau_n}\bigg\{ P_{+,t}(Y_{t-}^+)^2\Big[| \hat v_{+,t}^{\top}\sigma_t|^2+2r+2 \hat v_{+,t}^{\top}\mu_t-2\int_{\cE}v_{+,t}^{\top}\beta_t\nu(\de)\Big]\\
&\quad\;+P_{+,t}\sum_{j=1}^{\ell}\int_{\cE}\Big[(Y_{t-}^+)^2((1+\hat v_{+,t}^{\top}\beta_{j,e})^+)^2+(Y_{t-}^-)^2((-1+\hat v_{-,t}^{\top}\beta_{j,e})^+)^2-(Y_{t-}^+)^2
\Big]\nu_j(\de)\\
&\quad\;+P_{-,t}(Y_{t-}^-)^2\Big[| \hat v_{-,t}^{\top}\sigma_t|^2+2r-2 \hat v_{-,t}^{\top}\mu_t+2\int_{\cE}v_{-,t}^{\top}\beta_t\nu(\de)\Big]\\
&\quad\;+P_{-,t}\sum_{j=1}^{\ell}\int_{\cE}\Big[(Y_{t-}^+)^2((1+\hat v_{+,t}^{\top}\beta_{j,e})^-)^2+(Y_{t-}^-)^2((-1+\hat v_{-,t}^{\top}\beta_{j,e})^-)^2-(Y_{t-}^-)^2
\Big]\nu_j(\de)\\
&\quad\;-P_{+,t}(Y_{t-}^+)^2H_{+}^*(t,P_{+,t},P_{-,t})-P_{-,t}(Y_{t-}^-)^2H_{-}^*(t,P_{+,t},P_{-,t})\bigg\}\dt\\
&=P_{+,0}(Y_{0}^+)^2+P_{-,0}(Y_{0}^-)^2.
\end{align*}
Sending $n\rightarrow\infty$ and applying Fatou's lemma and $P_{\pm,T}=1$, we get
\begin{align*}
\E[Y_T^2]&\leq P_{+,0}(y^+)^2+P_{-,0}(y^-)^2,
\end{align*}
that is,
\begin{align*}
\E[(X^{\pi^*}_T-d)^2]&\leq P_{+,0}\big[(x-d)^+\big]^2+P_{-,0}\big[(x-d)^-\big]^2=V(0,x;d).
\end{align*}
The claim \eqref{verify} follows since the reverse inequality is trivial.
\end{proof}

\subsection{Efficient portfolio and efficient frontier to \eqref{optm}.}
\begin{lemma}\label{lemmap}
Let $(P_{+},P_{-})$ be given in Theorem \ref{existence}.
Then $P_{-,0}<1$.
\end{lemma}
\begin{proof}
Since $\beta$ is bounded, we can choose a sufficiently small constant $\varepsilon>0$ such that
\begin{align*}
1-v^{\top}\beta_{j}> 0, \ \mbox{for all $j$, $v\in\R^m_+$ with $|v|\leq \varepsilon$}.
\end{align*}
Recalling \eqref{def:Hv2} and using the boundedness of coefficients, there exists a constant $c>0$ such that
\begin{align*}
H_{-}(t,v,P_{+,t},P_{-,t})
&\leq P_{-,t}v^{\top}\Sigma_tv-2 P_{-,t}\mu_t^{\top}v \leq P_{-,t}\big(c|v|^2-2\mu_t^{\top}v\big),
\end{align*}
for all $v\in\R^m_+$ with $|v|\leq \varepsilon$.

On the other hand, the condition \eqref{feasible} implies, for sufficiently small $\varepsilon>0$,
that there exist one $i\in M$ and $\mathcal{O}\subseteq [0,T]$ with positive Lebesgue measure such that $\mu_{i,t}>c\varepsilon$ on $\mathcal{O}$. Setting $\hat v_{i}=\varepsilon$ and $\hat v=(0,\cdots,0, \hat v_{i}, 0,\ldots,0)\in\R^m_+$, recalling $P_{-,t}>0$, we have
\[
H_{-}^*(t,P_{+,t},P_{-,t})\leq H_{-}(t,\hat v,P_{+,t},P_{-,t})\leq P_{-,t}\big(c|\hat v|^2-2\mu_t^{\top}\hat v\big)< 0,~~t\in \mathcal{O}.
\]
Since $H_{-}^*(t,P_{+,t},P_{-,t})\leq 0$ for any $t\in[0,T]$, we conclude from \eqref{Riccati} that $$P_{-,0}=P_{-,T}+\int_{0}^{T} H_{-}^*(t,P_{+,t},P_{-,t})\dt<P_{-,T}=1.$$
The proof is complete.
\end{proof}

According to Proposition \ref{verifi},
\begin{align*}
V(0,x;d)=P_{+,0}[(x-d)^+]^2+P_{-,0}[(x-d)^-]^2
\end{align*}
is a piecewise quadratic function. A simple calculation shows
\begin{align*}
\sup_{d\in\R}[V(0,x;d)-(d-z)^{2}]=V(0,x;d^*)-(d^*-z)^{2}=\frac{P_{-,0}}{1-P_{-,0}}(z-x)^2,
\end{align*}
where
\begin{align*}
d^*=\frac{z-xP_{-,0}}{1-P_{-,0}}.
\end{align*}
(Note that in the calculation we made use of the fact that
\begin{align*}
0<P_{-,0}<1 \ \mbox{and } \ x-d^*=\frac{x-z}{1-P_{-,0}}<0,
\end{align*}
due to Lemma \ref{lemmap} and $z>x$.)

According to the Lagrange duality relationship \eqref{duality},
the above analysis boils down to the following solution to the MV problem \eqref{optm}.
\begin{thm}\label{Th:efficient}
Let $(P_{+},P_{-})$, $\hat v_{\pm}(t,P_{+,t},P_{-,t})$ be given in Theorem \ref{existence}.
For $z\geq x$, set
\begin{align*}
d^*=\frac{z-xP_{-,0}}{1-P_{-,0}}.
\end{align*}
Then the state feedback control defined by
\begin{align}\label{efficient}
\pi^*(t,X)=\hat v_{+}(t,P_{+,t},P_{-,t}) (X_{t-}-d^*)^++\hat v_{-}(t,P_{+,t},P_{-,t})(X_{t-}-d^*)^-,
\end{align}
is optimal for the MV problem \eqref{optm}.
Moreover, the efficient frontier to the problem \eqref{optm} is a half-line, determined by
\begin{align*}
\mathrm{Var}(X^{\pi^{*}}_T)=\frac{P_{-,0}}{1-P_{-,0}}\Big(\E[X_T^{\pi^{*}}]-x\Big)^2,
\end{align*}
where $\E[X_T^{\pi^{*}}]=z\geq x$.
\end{thm}
Since the interest rate is deterministic, the efficient frontier is as expected a half-line.
When the interest rate is random, then the minimum risk cannot be reduced to 0, so the efficient frontier is no more a half-line; see, e.g., \cite{HSX}. Of course, the HJB question approach is hardly to apply in that case.

\section{Proof of Theorem \ref{existence}.}\label{sec:Solvability}

This whole section is devoted to the proof of Theorem \ref{existence}. We shall use $c$ to represent a generic positive constant which can be different from line to line.

For any $(P_{+},P_{-})\in (0,\infty)\times (0,\infty)$, by Assumption \ref{assu1}, there are constants $c_1(P_{+},P_{-})>0$ and $c_2(P_{+},P_{-})>0$ such that
\begin{align}\label{bound}
H_{\pm}(t,v,P_{+},P_{-})\geq c_1|v|^2-c_2(|v|+1).
\end{align}
Hence $H_{\pm}(t,v,P_{+},P_{-})>0\geq H_{\pm}^*(t,P_{+},P_{-})$ whenever
$|v|$ is sufficiently large in terms of $c_{1}$ and $c_{2}$. 
Therefore, $H_{\pm}^*(t,P_{+},P_{-})$ are finite and locally Lipschitz w.r.t. $(P_{+},P_{-})$.

Assumption \ref{assu1} implies there is a constant $\alpha$ such that
$0<\alpha<\exp\big(-\int_0^T \mu_s^{\top}\Sigma_s^{-1}\mu_s\ds\big)$.
Let $g(P)=\alpha\vee(P\wedge 1)$. Then both $H_{\pm}^*(t,g(P_{+}),g(P_{-}))$ are Lipschitz continuous w.r.t. $(P_{+},P_{-})$ on $\R^{2}$, so the ODE 
\begin{align}\label{Riccatitrun}
\begin{cases}
\dot{P}_{+,t}=-H_{+}^*(t,g(P_{+,t}),g(P_{-,t})),\\
\dot{P}_{-,t}=-H_{-}^*(t,g(P_{+,t}),g(P_{-,t})),\\
P_{\pm,T}=1,
\end{cases}
\end{align}
admits a unique classical solution, denoted by $(\tilde P_{+},\tilde P_{-})$.
Since $$H_{\pm}^*(t,g(\tilde P_{+,t}),g(\tilde P_{-,t}))\leq H_{\pm}(t,0,g(\tilde P_{+,t}),g(\tilde P_{-,t}))=0$$ for all $t\in[0,T]$, we must have $\tilde P_{\pm,t}\leq \tilde P_{\pm,T}=1$.
If we can prove
\begin{align}\label{lower}
\tilde P_{\pm,t}\geq \alpha, \ t\in[0,T],
\end{align}
then $(\tilde P_{+},\tilde P_{-})$ is actually a positive solution to the ODE \eqref{Riccati} since $g(\tilde P_{\pm,t})=\tilde P_{\pm,t}$ in \eqref{Riccatitrun}.

To prove the estimate \eqref{lower},
we notice $\underline P_t=\exp\big(-\int_t^T \mu_s^{\top}\Sigma_s^{-1}\mu_s\ds)$, $t\in[0,T]$ satisfies the following ODE
\begin{align*}
\begin{cases}
\dot{\underline P}_t=\underline P_{t}\mu_t^{\top}\Sigma_t^{-1}\mu_t,\\
\underline P_T=1.
\end{cases}
\end{align*}
Clearly, $\alpha\leq\underline P_t\leq 1$, so we have
\begin{align}
-\underline P_{t}\mu_t^{\top}\Sigma_t^{-1}\mu_t&=\inf_{v\in\R^m}H_{+}(t,v,\underline P_{t},\underline P_{t})\leq \inf_{v\in\R^m_+}H_{+}(t,v,\underline P_{t},\underline P_{t})=H_{+}^*(t,g(\underline P_{t}),g(\underline P_{t})).
\end{align}
It follows from the chain rule that
\begin{align*}\label{chain}
[(\underline P_t-\tilde P_{+,t})^+]^2&=\int_t^T 2(\underline P_s-\tilde P_{+,s})^+\Big[-\underline P\mu_s^{\top}\Sigma_s^{-1}\mu_s-H_{+}^*(s,g(\tilde P_{+,s}),g(\tilde P_{-,s}))\Big]\ds\nn\\
&\leq \int_t^T 2(\underline P_s-P_{+,s})^+\Big[H_{+}^*(s,g(\underline P_{s}),g(\underline P_{s}))-H_{+}^*(s,g(\tilde P_{+,s}),g(\tilde P_{-,s}))\Big]\ds.
\end{align*}

On the other hand, using \eqref{bound}, we have
\begin{align*}\label{bounddomain}
H_{+}^*(t,g(\tilde P_{+,t}),g(\tilde P_{-,t}))=\inf_{v\in\R^m_+,|v|\leq c}H_{+}(t,v,g(\tilde P_{+,t}),g(\tilde P_{-,t})).
\end{align*}
Observe that $H_{+}(t,v,P_{+},P_{-})$ is non-decreasing w.r.t. $P_{-}$, we have
\begin{align*}
&\quad\;[(\underline P_t-\tilde P_{+,t})^+]^2\\
&\leq \int_t^T 2(\underline P_s-\tilde P_{+,s})^+\Big[\sup_{v\in\R^m_+,|v|\leq c}\Big(H_{+}(s,v,g(\underline P_{s}),g(\underline P_{s}))-H_{+}(s,v,g(\tilde P_{+,s}),g(\tilde P_{-,s}))\Big)\Big]\ds\\
&\leq c\int_t^T (\underline P_s-\tilde P_{+,s})^+\Big[|\underline P_s-\tilde P_{+,s}|+(\underline P_s-\tilde P_{-,s})^+\Big]\ds\nn\\
&\leq c\int_t^T[(\underline P_s-\tilde P_{+,s})^+]^2+[(\underline P_s-\tilde P_{-,s})^+]^2\ds.
\end{align*}
Similarly, we have
\begin{align*}
[(\underline P_t-\tilde P_{-,t})^+]^2
&\leq c\int_t^T[(\underline P_s-\tilde P_{+,s})^+]^2+[(\underline P_s-\tilde P_{-,s})^+]^2\ds.
\end{align*}
Combining the above two inequalities gives
\begin{align*}
[(\underline P_t-\tilde P_{+,t})^+]^2+[(\underline P_t-\tilde P_{-,t})^+]^2
&\leq c\int_t^T[(\underline P_s-\tilde P_{+,s})^+]^2+[(\underline P_s-\tilde P_{-,s})^+]^2\ds.
\end{align*}
We infer from Gronwall's inequality that $[(\underline P_t-\tilde P_{+,t})^+]^2+[(\underline P_t-\tilde P_{-,t})^+]^2=0$, which implies $\tilde P_{\pm,t}\geq \underline P_t\geq \alpha$, establishing \eqref{lower}.

Suppose \eqref{Riccati} admits two positive solutions $(P_{+}, P_{-})$ and $(P'_{+}, P'_{-})$. Since $H^{*}_{\pm}\leq 0$, both of the solutions are increasing. Hence, there exists a sufficiently small constant $\alpha>0$ such that $\alpha\leq P_{\pm},P'_{\pm}\leq 1$. Then both $(P_{+}, P_{-})$ and $(P'_{+}, P'_{-})$ are solutions to the ODE \eqref{Riccatitrun} with Lipschitz driver, so they must be equal.

The existence and boundedness of $\hat v_{\pm}(t,P_{+,t},P_{-,t})$ follow from \eqref{bound} and
the boundedness of $P_{\pm,t}$.

\section{Examples with explicit solutions to \eqref{Riccati}.}\label{sec:examples}
In this section, we present some special cases in which the ODE \eqref{Riccati} can be decoupled.

\subsection{Case $\beta_{ij}\geq0$.}
If $\beta_{ij}\geq 0$ for all $i,j$, then $1+v^{\top}\beta_{j}\geq 0$ for any $v\in\R^m_+$. Moreover,
\begin{align*}
H_{+}^*(t,P_{+},P_{-})&=\inf_{v\in\R^m_+}H_{+}(t,v,P_{+},P_{-})=P_{+}\inf_{v\in\R^m_+}\Big(v^{\top}\Sigma_t v+2\mu_t^{\top} v\Big)
\end{align*}
is linear w.r.t. $P_{+}$ and independent of $P_{-}$. In this case, the two-dimensional ODE \eqref{Riccati} is partially decoupled. One can first solve $P_{+}$ and then $P_{-}$.
If, furthermore, one has $\mu_{i}\geq 0$ for all $i$, then
\begin{align}
H_{+}^*(t,P_{+},P_{-})\equiv 0, ~~~ P_{+,t}\equiv 1, ~~~ \hat v_{+}(t,P_{+,t},P_{-,t})\equiv0.
\end{align}

\begin{assmp}\label{assu2}
All the coefficients $\mu$, $\sigma$, $\beta$ are time-invariant and positive.
Meanwhile, $m=n=\ell=1$, $\cE=\{1\}$ and $\nu(\cE)=1$.
\end{assmp}
This assumption says that there is only one risky asset, both the Brownian motion $W$ and the Poisson random measure $N$ are one-dimensional, and the Poisson random measure $N$ degenerates to a Poisson process with intensity $1$.
Let Assumption \ref{assu2} hold in the remaining of this subsection.

We now obtain analytic expressions for $H_{-}(v,P_{+},P_{-})$, $H^*_{-}(P_{+},P_{-})$, $\hat v_{-}(P_{+},P_{-})$ (here we omit the argument $t$ since they are time independent). In this case,
\begin{align*}
H_{-}(v,P_{+},P_{-})&=P_{-}\sigma^2v^2-2P_{-}\mu v+P_{-}\Big[\big[(1-\beta v)^+\big]^2-1+2\beta v\Big]+\big[(1-\beta v)^{-}\big]^2\\[3pt]
&=\begin{cases}
P_{-}(\sigma^2+\beta^2)v^2-2P_{-}\mu v, &\ \mbox{if} \ 0\leq v\leq \frac{1}{\beta};\bigskip\\
(P_{-}\sigma^2+\beta^2)v^2-2(P_{-}\mu-P_{-}\beta+\beta)v+1-P_{-}, & \ \mbox{if} \ v\geq \frac{1}{\beta}.
\end{cases}
\end{align*}
A simple calculation shows
\begin{align*}
&\quad \inf_{0\leq v\leq \beta^{-1}}\Big[P_{-}(\sigma^2+\beta^2)v^2-2P_{-}\mu v\Big]\\[3pt]
&=
\begin{cases}
H_{-}(\frac{1}{\beta},P_{+},P_{-} )=-\frac{2\beta\mu-\sigma^2-\beta^2}{\beta^2}P_{-}, & \mbox{if} \ \sigma^2+\beta^2\leq \beta\mu;\bigskip\\
H_{-}(\frac{\mu}{\sigma^2+\beta^2},P_{+},P_{-} )=-\frac{\mu^2}{\sigma^2+\beta^2}P_{-}, &\mbox{if} \ \sigma^2+\beta^2\geq \beta\mu,
\end{cases}
\end{align*}
and
\begin{align*}
&\quad \inf_{v\geq \beta^{-1}}\Big[(P_{-}\sigma^2+\beta^2)v^2-2(P_{-}\mu-P_{-}\beta+\beta)v+1-P_{-}\Big]\\[3pt]
&=
\begin{cases}
H_{-}(\frac{P_{-}\mu-P_{-}\beta+\beta}{P_{-}\sigma^2+\beta^2},P_{+},P_{-} )=-\frac{(\mu P_{-}-\beta P_{-}+\beta)^2}{\sigma^2P_{-}+\beta^2}+1-P_{-}, &\ \mbox{if} \ \sigma^2+\beta^2\leq \beta\mu;\bigskip\\
H_{-}(\frac{1}{\beta},P_{+},P_{-} )=-\frac{2\beta\mu-\sigma^2-\beta^2}{\beta^2}P_{-}, & \ \mbox{if} \ \sigma^2+\beta^2\geq \beta\mu.
\end{cases}
\end{align*}
Therefore,
\begin{align*}
H_{-}^*(P_{+},P_{-})&=\min\Big\{\inf_{0\leq v\leq \beta^{-1}}\Big[P_{-}(\sigma^2+\beta^2)v^2-2P_{-}\mu v\Big],\\
&\qquad\qquad\;\inf_{v\geq \beta^{-1}}\Big[(P_{-}\sigma^2+\beta^2)v^2-2(P_{-}\mu-P_{-}\beta+\beta)v+1-P_{-}\Big]\Big\}\bigskip\\
&=
\begin{cases}
H_{-}(\frac{P_{-}\mu-P_{-}\beta+\beta}{P_{-}\sigma^2+\beta^2},P_{+},P_{-} )=-\frac{(\mu P_{-}-\beta P_{-}+\beta)^2}{\sigma^2P_{-}+\beta^2}+1-P_{-}, \\
\hspace{14em} \mbox{if} \ \sigma^2+\beta^2\leq \beta\mu
\Longleftrightarrow \frac{P_{-}\mu-P_{-}\beta+\beta}{P_{-}\sigma^2+\beta^2}\geq \frac{1}{\beta};\bigskip\\
H_{-}(\frac{\mu}{\sigma^2+\beta^2},P_{+},P_{-} )=-\frac{\mu^2}{\sigma^2+\beta^2}P_{-}, \\ \hspace{14em}\mbox{if} \ \sigma^2+\beta^2\geq \beta\mu\Longleftrightarrow \frac{\mu}{\sigma^2+\beta^2}\leq \frac{1}{\beta},
\end{cases}
\end{align*}
and
\begin{align*}
\hat v_{-}(P_{+},P_{-})=
\begin{cases}
\frac{P_{-}\mu-P_{-}\beta+\beta}{P_{-}\sigma^2+\beta^2}, &\ \mbox{if} \ \sigma^2+\beta^2\leq \beta\mu;\\
\frac{\mu}{\sigma^2+\beta^2}, &\ \mbox{if} \ \sigma^2+\beta^2\geq \beta\mu.
\end{cases}
\end{align*}

Note that $\hat v_{+}(P_{+},P_{-})\equiv 0$ under Assumption \ref{assu2}, so according to Theorem \ref{Th:efficient},
the efficient feedback portfolio is reduced to
\begin{align}\label{pistar2}
\pi^*(t,X)=\hat v_{-}(P_{+,t},P_{-,t})(X_{t-}-d^* )^-.
\end{align}

Substituting \eqref{pistar2} into \eqref{wealth}, the optimal state process $X$ follows
\begin{align}\label{shiftedstate}
\begin{cases}
\dd \ (X_t-d^*)=(X_{t-}-d^*)^-\hat v_{-} (\mu \dt+\sigma \dw_t+\beta \dd\tilde N_t),\\
X_0-d^*=x-d^*<0,
\end{cases}
\end{align}
where $\hat v_{-}=\hat v_{-}(P_{+},P_{-})$.

\begin{description}
\item[Case $\sigma^2+\beta^2< \beta\mu$.]
In this case, 
$1-\beta \hat v_{-}=1-\beta\frac{P_{-}\mu-P_{-}\beta+\beta}{P_{-}\sigma^2+\beta^2}
=\frac{\sigma^2+\beta^2-\beta\mu}{P_{-}\sigma^2+\beta^2}P_{-}< 0$.
Denote by $\tau_k$ the $k$-th jump time of the Poisson process $N$ with the convention that $\tau_0=0$. Then the unique solution of \eqref{shiftedstate} is given by
\begin{align*}
X_t-d^*&=(x-d^*)(1-\beta\hat v_{-})^{2k}\Big(\prod_{i=0}^{k-1}L_{\tau_{2i},\tau_{2i+1}}\Big)\\
&\qquad\times\bigg[L_{\tau_{2k}, t}\mathbf{1}_{\tau_{2k}\leq t<\tau_{2k+1}} 
+(1-\beta\hat v_{-})L_{\tau_{2k}, \tau_{2k+1}}\mathbf{1}_{\tau_{2k+1}\leq t<\tau_{2k+2}}\bigg],
\end{align*}
where
\begin{align}\label{L}
L_{s,t}:=\exp\Big(-(\mu\hat v_{-}-\beta\hat v_{-}+\frac{1}{2}\sigma^2\hat v_{-}^2)(t-s)-\sigma\hat v_{-}(W_t-W_s)\Big), \ s\leq t.
\end{align}
For each $k=0,1,\ldots$, we have $X_t-d^*<0$ for $t\in[\tau_{2k},\tau_{2k+1})$, and $X_t-d^*>0$ is constant for $t\in[\tau_{2k+1},\tau_{2k+2})$. That is to say, at every jump time $\tau_{k}$ of the Poisson process $N$, the process $X_t-d^*$ changes its sign.
Therefore, the optimal portfolio is
\begin{align*}
\pi^*(t,X)=-\frac{P_{-,t}\mu-P_{-,t}\beta+\beta}{P_{-,t}\sigma^2+\beta^2}
(X_{t-}-d^*)\mathbf{1}_{t\in\cup_{k=0}^{\infty}[\tau_{2k},\tau_{2k+1})}, \ t\in[0,T].
\end{align*}
Figure \ref{figure:x} demonstrates the dynamics $X_{t}-d^{*}$ in the case $\sigma^2+\beta^2< \beta\mu$.
\begin{figure}[H]
\begin{minipage}[t]{\textwidth}
\centering
\includegraphics[width=0.75\linewidth]{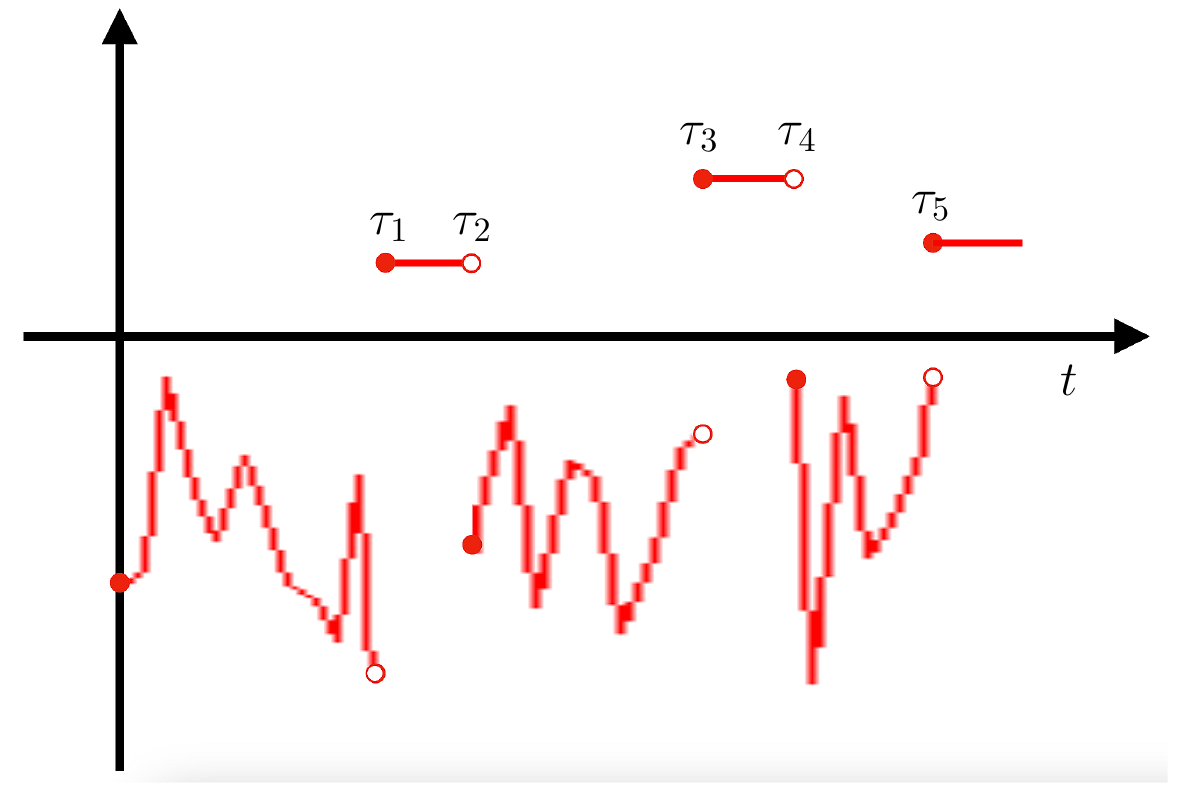}
\caption{The dynamics $X_{t}-d^{*}$ in \eqref{shiftedstate} in the case $\sigma^2+\beta^2< \beta\mu$.}\medskip\label{figure:x}
\end{minipage}
\end{figure}
\item[Case $\sigma^2+\beta^2\geq \beta\mu$.]
In this case we have $1-\beta \hat v=1-\beta\frac{\mu}{\sigma^2+\beta^2}\geq 0$, then
\begin{align*}
X_t-d^*=(x-d^*)L_{0,t}\prod_{0<s\leq t}((1-\beta \hat v_{-})\Delta N_s)\leq0, \ t\geq 0,
\end{align*}
so the optimal portfolio is
\begin{align*}
\pi^*(t,X)=-\frac{\mu}{\sigma^2+\beta^2}(X_{t-}-d^*)\geq 0, \ t\in[0,T].
\end{align*}

\end{description}

\subsection{Case $-1<\beta_{ij}\leq0$.}

If $-1<\beta_{ij}\leq 0$ for all $i,j$, then $1-v^{\top}\beta_{j}\geq 0$ for any $v\in\R^m_+$. Moreover,
\begin{align*}
H_{-}^*(t,P_{+},P_{-})&=\inf_{v\in\R^m_+}H_{-}(t,v,P_{+},P_{-})=P_{-}\inf_{v\in\R^m_+}\Big(v^{\top}\Sigma_t v-2\mu_t^{\top} v\Big)
\end{align*}
is linear w.r.t. $P_{-}$, independent of $P_{+}$. In this case, the two-dimensional ODE \eqref{Riccati} is partially decoupled. One can first solve $P_{-}$ and then $P_{+}$. And
$$\hat v_{-}(t,P_{+},P_{-})=\argmin_{v\in\R^m_+} \Big(v^{\top}\Sigma_t v-2\mu_t^{\top} v \Big),$$
is independent of $(P_{+},P_{-})$.

\begin{assmp}\label{assu3}
Same as Assumption \ref{assu2} except for $\beta>0$ replaced by $-1<\beta\leq 0$.
\end{assmp}
Let Assumption \ref{assu3} hold in the remaining of this subsection.

In this case,
\begin{align*}
\frac{\partial H_{+}(v,P_{+},P_{-})}{\partial v}=2P_{+}\sigma^2v+2P_{+}\mu+P_{+}[2\beta(1+\beta v)^+-2\beta]-2P_{-}\beta(1+\beta v)^{-}
\end{align*}
is nondecreasing w.r.t $v$. Therefore for any $v\in\R_+$,
\begin{align*}
\frac{\partial H_{+}(v,P_{+},P_{-})}{\partial v}\geq\frac{\partial H_{+}(v,P_{+},P_{-})}{\partial v}\Big|_{v=0}=2P_{+}\mu>0,
\end{align*}
we know that $H_{+}(v,P_{+},P_{-})$ is nondecreasing w.r.t $v\in\R_+$.
Therefore
\begin{align*}
H_{+}^*(P_{+},P_{-})=H_{+}(0,P_{+},P_{-})=0,
\end{align*}
and
\begin{align*}
P_{+,t}\equiv1, ~~~\hat v_{+}(P_{+,t},P_{-,t})\equiv0.
\end{align*}

On the other hand, we have
\begin{align*}
H_{-}(v,P_{+},P_{-})&=P_{-}\sigma^2v^2-2P_{-}\mu v+P_{-}[((1-\beta v)^+)^2-1+2\beta v]\\
&=P_{-}[(\sigma^2+\beta^2)v^2-2\mu v],
\end{align*}
as $1-\beta v\geq0$ for any $v\geq0$.
Moreover
\begin{align*}
H_{-}^*(P_{+},P_{-})=\inf_{v\geq0}H_{-}(v,P_{+},P_{-})=H_{-}\Big(\frac{\mu}{\sigma^2+\beta^2},P_{+},P_{-}\Big)
=-\frac{\mu^2}{\sigma^2+\beta^2}P_{-},
\end{align*}
and
\begin{align}\label{hatv2ex}
\hat v_{-}(P_{+},P_{-})=\frac{\mu}{\sigma^2+\beta^2}.
\end{align}

According to Theorem \ref{Th:efficient},
the efficient feedback portfolio turns to
\begin{align}\label{pistar3}
\pi^*(t,X)=\frac{\mu}{\sigma^2+\beta^2}(X_{t-}-d^*)^-.
\end{align}
Substituting \eqref{pistar3} into \eqref{wealth}, we obtain
\begin{align*}
\begin{cases}
\dd \ (X_t-d^*)=(X_{t-}-d^*)^-(\mu \hat v_{-}\dt+\sigma \hat v_{-}\dw_s+\beta \hat v_{-}\dd\tilde N_t),\\
X_0-d^*=x-d^*<0.
\end{cases}
\end{align*}
Since $1-\beta\hat v_{-}=1-\beta\frac{\mu}{\sigma^2+\beta^2}> 0$, we see
\begin{align*}
X_t-d^*=(x-d^*)L_{0,t}\prod_{0<s\leq t}((1-\beta \hat v_{-})\Delta N_s)<0, ~~t\geq 0,
\end{align*}
where $L$ and $\hat v_{-}$ are given in \eqref{L} and \eqref{hatv2ex} respectively.

\end{document}